\documentclass[11pt,reqno]{amsart}

\usepackage{mathrsfs,textcomp,mathtools,accents}
\usepackage{hyperref}

\usepackage[babel]{csquotes}
\usepackage[
  backend=bibtex8,
  citestyle=numeric,
  useprefix,
  firstinits=true,
  uniquename=init,
  maxcitenames=1,
  maxbibnames=20,
  hyperref]{biblatex}

\renewcommand*{\bibpagespunct}{\addcolon\space}  

\DeclareFieldFormat[article,inbook,incollection,inproceedings,patent,thesis,unpublished]{title}{\mkbibemph{#1}\isdot}

\DeclareFieldFormat[article]{journaltitle}{#1\isdot}

\DeclareFieldFormat[article]{pages}{#1}

\renewbibmacro{in:}{%
  \ifentrytype{article}{}{\printtext{\bibstring{in}\intitlepunct}}}
\DeclareFieldFormat[article,periodical]{volume}{\mkbibbold{#1}}
\DeclareFieldFormat[article,periodical]{number}{\mkbibparens{#1}
}
\renewbibmacro*{volume+number+eid}{%
  \printfield{volume}%
  \printfield{number}%
  \setunit{\bibpagespunct}%
  \printfield{pages}}

\renewbibmacro*{issue+date}{%
  \setunit{\addcomma\space}%
  \iffieldundef{issue}
  {\usebibmacro{date}}
  {\printfield{issue}%
    \setunit*{\addspace}%
    \usebibmacro{date}}%
  \newunit
}

\renewbibmacro*{note+pages}{%
  \printfield{note}%
  \setunit{\bibpagespunct}%
  \newunit
}

\usepackage{fullpage}

\usepackage{enumerate}
\usepackage[shortlabels]{enumitem}
\setenumerate{label={\upshape(\roman*)},
  wide=0pt,topsep=0.2cm,labelsep=0.1em,itemsep=0.15cm,leftmargin=*}

\newtheorem{theorem}{Theorem}

\newtheorem{lemma}[theorem]{Lemma}

\theoremstyle{definition}

\renewcommand{\phi}{\varphi}
\renewcommand{\epsilon}{\varepsilon}

\newcommand{\dd}{\,\mathrm{d}}

\usepackage{bbm}

\bibliography{biblio.bib}

\begin{document}

\title{On a solution to the Basel problem based on the fundamental
  theorem of calculus}

\author{Alessio Del Vigna}


\maketitle


\section{Introduction} \label{sec:intro}

Let $s$ be a real number. Everyone knows that the infinite series
\[
    \sum_{n=1}^\infty \frac{1}{n^s}
\]
is convergent if $s>1$ and that it diverges if $s\leq 1$. The sum of
this series is denoted by $\zeta(s)$ and it is known as the
\emph{Riemann zeta function}\footnote{The Riemann zeta function is
  actually defined for $s$ being complex. In this setting it can be
  proven that the series $\sum_{n=1}^\infty {1}/{n^s}$ converges if
  and only if Re$(s)>1$ and that $\zeta(s)$ can be extended to a
  meromorphic function on the whole complex plane, which is
  holomorphic everywhere but a simple pole at $s=1$.}. In 1743 Euler
computed the value of $\zeta(2)$ by proving the identity
\begin{equation}\label{zeta}
    \sum_{n=1}^\infty \frac{1}{n^2} = \frac{\pi^2}6.
\end{equation}
The problem of evaluating the sum of the reciprocals of the squares,
also known under the name of \emph{Basel problem}, was first posed by
Mengoli in the mid-seventeenth century and attacked by many prominent
mathematicians of the time without success, until Euler.

Over the years, several solutions to the Basel problem have been found
using a vast variety of techniques. The proof in \cite{matsuoka}
relies on a recurrence equation obtained by cleverly evaluating
certain trigonometric integrals and then by telescoping. Some other
proofs are based on evaluations of double integrals: the proof by
Apostol \cite{apostol} and that by Beukers, Kolk, and Calabi
\cite{BKC} use simple double integral and ingenious substitutions, for
which they deserved a place in the wonderful text ``Proofs from {\sc
  the book}'' \cite{proofs-book}. The more recent proof in
\cite{ritelli} uses the double integral of a rational function with
the lowest degree among the functions used in other similar
proofs. Moreover, many textbooks in Fourier analysis contain proofs or
guided exercises about the Basel problem, based on the evaluation of
the Fourier series of certain functions or on the Parseval
identity. In some complex analysis textbooks the proof is instead
given via contour integrals and their evaluation through the residue
theorem. Lastly, \cite{pace} contains an original proof based on
elementary probability tools. This list is not intended to be
exhaustive, but it is just a way to show that original proofs can come
from different area of mathematics.

Here we give another way of solving the Basel problem from the area of
the mathematical analysis. The main ingredients for our proof are the
differentiation under the integral sign, a trick which is recurrent in
series evaluation, and the fundamental theorem of calculus, which to
our knowledge has never been used in this context.


\section{The proof} \label{sec:proof}

We consider the function
$f:(0,\pi/2)\times [0,1] \rightarrow \mathbb{R}$ defined to be
\[
    f(x,t) = \arccos \bigg(\frac{t-\tan^2 x}{t+\tan^2 x}\bigg).
\]
Since the integral of $f(x,t)$ with respect to $x$ exists for every
$t\in [0,1]$, we are allowed to define $g:[0,1]\rightarrow \mathbb{R}$,
function of just the variable $t$, to be
\[
    g(t) = \int_0^{\frac{\pi}2} \arccos\bigg(\frac{t-\tan^2
        x}{t+\tan^2 x}\bigg) \dd x.
\]
Here it comes the differentiation under the integral sign because we
would like to compute the derivative of the function $g$.

\begin{lemma}\label{lemma:g_diff}
    The function $g$ is differentiable on $(0,1)$ and it holds
    \[
        g'(t) = \frac{\log t}{2\sqrt{t}(1-t)}.
    \]
\end{lemma}
\begin{proof}
    The function $g$ is defined through an integral and the conclusion
    is just a computation provided that we are allowed to
    differentiate under the integral sign. To do this it suffices to
    show that the conditions stated in Theorem~\ref{thm:diff_int}
    hold. For all $x\in (0,\pi/2)$ and for all
    $t\in (0,1)$ we have the existence of the partial derivative
    \[
        \frac{\partial f}{\partial t}(x,t) = -\frac{\tan
          x}{\sqrt{t}(t+\tan^2 x)}.
    \]
    For the domination condition we can use the inequality between the
    arithmetic and the geometric mean of non-negative numbers:
    \[
        \bigg|\frac{\partial f}{\partial t}(x,t)\bigg| =
        \frac{1}{t} \cdot \frac{\sqrt{t}\tan x}{t+\tan^2x}\leq
        \frac{1}{2t}.
    \]
    We then fix $\delta$ with $0<\delta<1$ and we restrict $t$ to the
    interval $(\delta,1)$, so that the previous bound yields
    \[
        \bigg|\frac{\partial f}{\partial t}(x,t)\bigg|\leq \frac{1}{2
          \delta},
    \]
    with the bounding function being integrable over $(0,\pi/2)$. Thus
    $g$ is differentiable on the interval $( \delta,1)$ for every
    $ \delta$, and since $0< \delta<1$ we have the differentiability
    of $g$ over the whole $(0,1)$. We are thus allowed to
    differentiate $g$ by passing the derivative under the integral
    sign. This yields
    \begin{align*}
        g'(t) &= -\frac{1}{\sqrt{t}}\int_0^{\frac{\pi}2} \frac{\tan
          x}{t+\tan^2x} \dd x =
        -\frac{1}{\sqrt{t}}\int_0^{\frac{\pi}2} \frac{\sin x\cos
          x}{(t-1)\cos^2x+1} \dd x=\\
        &= \frac{1}{2\sqrt{t}(t-1)}\log((t-1)\cos 2x +
        t+1)\big|_{x=0}^{x=\frac{\pi}2}=
        \frac{\log t}{2\sqrt{t}(1-t)}.
    \end{align*}
    and Lemma~\ref{lemma:g_diff} is proved.
\end{proof}

We are now in a position to prove \eqref{zeta} by applying the
fundamental theorem of calculus to the function $g'$ over the interval
$(0,1)$. As in many solutions to the Basel problem, we do not directly
show \eqref{zeta} but the equivalent
\begin{equation}\label{equiv}
    \sum_{n=0}^\infty \frac{1}{(2n+1)^2} = \frac{\pi^2}{8}.
\end{equation}
Indeed we immediately have
$\sum_{n=0}^\infty \frac{1}{(2n+1)^2} =\frac 34 \sum_{n=1}^\infty
\frac{1}{n^2}$.

\begin{theorem}
    It holds that
    \[
        \sum_{n=1}^\infty\frac{1}{n^2} = \frac{\pi^2}{6}.
    \]
\end{theorem}
\begin{proof}
    In Lemma~\ref{lemma:g_diff} we proved that $g$ is differentiable
    on $(0,1)$, which implies that $g$ is continuous on $(0,1)$. We
    claim that $g$ is also continuous in $t=0$. Let $(t_n)_{n=0}^\infty$
    be a sequence in $(0,1)$ such that $t_n\rightarrow 0$. Since for
    all $x\in (0,\pi/2)$ and for all $t\in [0,1]$ holds
    $|f(x,t)|\leq \pi$ and since $(0,\pi/2)\times [0,1]$ has
      finite measure we can apply the dominated convergence theorem to
      obtain
    \[
        \lim_{n\rightarrow +\infty} g(t_n) = \lim_{n\rightarrow
          +\infty}\int_0^{\frac{\pi}2} f(x,t_n) \dd x =
        \int_0^{\frac{\pi}2} \lim_{n\rightarrow+\infty} f(x,t_n) \dd x
        = g(0).
    \]
    An analogous argument holds for $t=1$, so that $g$ turns out to be
    continuous on the whole $[0,1]$. We now claim that $g'$ is
    integrable on $(0,1)$: $g'$ is continuous on $(0,1)$, it can be
    extended by continuity in $t=1$, and moreover
    \[
        \frac{\log t}{\sqrt{t}(1-t)}\sim \frac{\log
          t}{\sqrt{t}}\quad\text{ for } t\rightarrow 0^+.
    \]
    Thus we can apply the fundamental theorem of calculus to $g'$ to
    write
    \begin{equation}\label{fund}
        g(1)-g(0) = \int_0^1 g'(t) \dd t.
    \end{equation}
    We now evaluate the quantities of the above identity. We have
    $g(0) = \int_0^{{\pi}/2} \pi \dd x = \pi^2/2$ and
    \[
        g(1) = \int_0^{\frac{\pi}2} \arccos \bigg(\frac{1-\tan^2
          x}{1+\tan^2 x}\bigg) \dd x = \int_0^{\frac{\pi}2} 2x \dd x =
        \frac{\pi^2}4,
    \]
    where we used the trigonometric identity
    $(1-\tan^2 x)/(1+\tan^2 x)=\cos 2x$. For the right-hand side
    of~\eqref{fund} we use the expression of $g'$ and series expansion
    to evaluate the resulting integral:
    \begin{align*}
        \int_0^1 g'(t) \dd t &=
        \frac 12 \int_0^1 \frac{\log t}{\sqrt{t}(1-t)} \dd t =
        2\int_0^1 \frac{\log u}{1-u^2} \dd u=\\
        &=2\int_0^1 \log u \Bigg(\sum_{n=0}^\infty u^{2n}\Bigg) \dd u =
        2\sum_{n=0}^\infty \bigg(\int_0^1 u^{2n}\log u \dd u\bigg) =\\
        &=-2 \sum_{n=0}^\infty \frac{1}{(2n+1)^2}.
    \end{align*}
    By equating the quantities involved in \eqref{fund}, we
    immediately obtain \eqref{equiv}.
\end{proof}

\appendix
\section{Differentiation under the integral sign}

Differentiation under the integral sign concerns integrals depending
on a parameter and it is often a powerful tool to evaluate definite
integrals and series. Over analysis and measure theory textbooks one
can find plenty of versions, with slight differences in the hypotheses
or in the conclusions. To avoid misunderstandings, we decided to
explicitly state the version we used: see, for instance,
\cite[Theorem~6.2.6]{hijab}.

\begin{theorem}\label{thm:diff_int}
    Let $X\subseteq \mathbb{R}^n$ be a measurable set and $A\subseteq \mathbb{R}$ be
    an open set. Let $f:X\times A\rightarrow \mathbb{R}$ be a measurable
    function such that
    \begin{enumerate}
      \item the function $\mathbf{x}\mapsto f(\mathbf{x},t)$ is
        integrable on $X$ for all $t\in A$;
      \item the partial derivative
        ${\partial f}/{\partial t}(\mathbf{x},t)$ exists for
        a.e. $\mathbf{x}\in X$ and for all $t\in A$;
      \item there exists an integrable function $h:X\rightarrow \mathbb{R}$
        such that
        $|{\partial f}/{\partial t}(\mathbf{x},t)|\leq
        h(\mathbf{x})$ for a.e. $\mathbf{x}\in X$ and for all
        $t\in A$.
    \end{enumerate}
    Then for all $t\in A$ holds
    \[
        \frac{ d}{ d t} \bigg(\int_X f(\mathbf{x},t)
          \dd \mathbf{x}\bigg) = \int_X \frac{\partial f}{\partial
          t}(\mathbf{x},t) \dd \mathbf{x}.
    \]
\end{theorem}
\begin{proof}
    Let $t\in A$ and let $r>0$ such that
    $B_r(t)= (t-r,t+r)\subseteq A$. Let $(t_n)_{n=0}^\infty$ be a
    sequence in $B_r(t)$ converging to $t$ and consider the ratio
    \[
        \frac{\int_X f(\mathbf{x},t_n) \dd \mathbf{x} -\int_X
          f(\mathbf{x},t) \dd \mathbf{x}}{t_n-t}=
        \int_X \frac{f(\mathbf{x},t_n)-f(\mathbf{x},t)}{t_n-t}
        \dd \mathbf{x}.
    \]
    Condition (ii) implies that the last integrand function converges
    to ${\partial f}/{\partial t}(\mathbf{x},t)$ for
    a.e. $\mathbf{x}\in X$. From the mean value theorem there exists
    $\xi\in B_r(t)$ such that for a.e. $\mathbf{x}\in X$ holds
    \[
        \frac{f(\mathbf{x},t_n)-f(\mathbf{x},t)}{t_n-t}=
        \frac{\partial f}{\partial t}(\mathbf{x},\xi),
    \]
    and hence condition (iii) implies that for a.e. $\mathbf{x}\in X$
    \[
        \bigg|\frac{f(\mathbf{x},t_n)-f(\mathbf{x},t)}{t_n-t}\bigg| =
        \bigg|\frac{\partial f}{\partial t}(\mathbf{x},\xi)\bigg|\leq
        h(\mathbf{x}).
    \]
    We can thus apply the Lebesgue dominated convergence theorem to
    obtain
    \begin{align*}
        \frac{ d}{ d t} \bigg(\int_X f(\mathbf{x},t)
          \dd \mathbf{x}\bigg) &=
        \lim_{n\rightarrow +\infty}\int_X
        \frac{f(\mathbf{x},t_n)-f(\mathbf{x},t)}{t_n-t} \dd \mathbf{x}=\\
        &=\int_X \bigg(\lim_{n\rightarrow +\infty}
          \frac{f(\mathbf{x},t_n)-f(\mathbf{x},t)}{t_n-t}\bigg)
        \dd \mathbf{x}=
        \int_X \frac{\partial f}{\partial
          t}(\mathbf{x},t) \dd \mathbf{x},
    \end{align*}
    so that Theorem~\ref{thm:diff_int} is proved. \end{proof}

\printbibliography

\end{document}